\documentclass[oneside,english]{amsart}
\usepackage{textcomp}
\usepackage{amsthm}
\usepackage{amssymb}
\usepackage{esint}

\makeatletter
\numberwithin{equation}{section}
\numberwithin{figure}{section}
\theoremstyle{plain}
\newtheorem{thm}{\protect\theoremname}
  \theoremstyle{definition}
  \newtheorem{defn}[thm]{\protect\definitionname}
  \theoremstyle{plain}
  \newtheorem{lem}[thm]{\protect\lemmaname}

\makeatother

\usepackage{babel}
  \providecommand{\definitionname}{Definition}
  \providecommand{\lemmaname}{Lemma}
\providecommand{\theoremname}{Theorem}

\begin{document}

\title{analysis of a model arising from invasion by precursor and differentiated
cells}

\author{xiaojie hou}

\curraddr{Department of Mathematics, University of North Carolina Wilmington,
Wilmington NC 28403, USA.}

\email{houx@uncw.edu}

\keywords{Traveling Wave Solution, Monotone iteration, Upper and Lower Solutions,
Existence, Uniqueness, Asymptotics.}

\subjclass[2000]{35K45 (92D30) }
\begin{abstract}
We study the wave solutions for a degenerated reaction diffusion system
arising from the invasion of cells. We show that there exists a family
of waves for the wave speed larger than or equals a certain number,
and below which there is no monotonic wave solutions. We also investigate
the monotonicity, uniqueness and asymptotics of the waves. 
\end{abstract}
\maketitle

\section{\label{sec:1}introduction}

In \cite{Trewnack}, the following coupled partial differential equation
system was proposed to study the invasion by precursor and differentiated
cells:

\begin{equation}
\left\{ \begin{array}{lll}
u_{t} & = & du_{xx}+\alpha u(1-\frac{u+\nu v}{k_{1}})-\beta u(1-\frac{v}{k_{2}}),\\
\\
v_{t} & = & \beta u(1-\frac{v}{k_{2}}),
\end{array}\right.\label{eq:1.1}
\end{equation}
where $u(x,t)$ denotes the population densities of the precursor
cells. The constant $d>0$ is the diffusion rate of the cell $u$
which has proliferation rate $\alpha>0$, and $k_{1}>0$ is the carrying
capacity of $u$. The parameter $\nu$ measures the relative contribution
that the differentiated cell with population density $v(x,t)$ makes
to the carrying capacity $k_{1}$. The cell population density $v$
is limited by its carrying capacity $k_{2}$ and has a max. differentation
rate $\beta>0$. The model assumes that the differentiated cells do
not have mobility.

Letting (see \cite{Trewnack})

\[
\hat{u}=\frac{u}{k_{1}},\,\hat{v}=\frac{v}{k_{1}},\,\hat{t}=\alpha t,\,\hat{x}=\sqrt{\frac{\alpha}{D}}x
\]
and dropping the hat notation for convenience, system (\ref{eq:1.1})
is changed into

\begin{equation}
\left\{ \begin{array}{lll}
u_{t} & = & u_{xx}+u(1-u-\nu v)-\lambda u(1-Kv),\\
\\
v_{t} & = & \lambda u(1-Kv),
\end{array}\right.\label{eq:1.2}
\end{equation}
where $\lambda=\frac{\beta}{\alpha}$ and $K=\frac{k_{1}}{k_{2}}$.

System (\ref{eq:1.1}) or (\ref{eq:1.2}) belongs to reaction diffusion
systems of degenerate type, and such systems have attracted much attention
in the research such as epidemics and wound healing \cite{XuZhao,ZhaoWang,Denman}.
However, system (\ref{eq:1.2}) differs from the above systems in
the appearance of degenerate reaction terms. In fact, $u=0$ coupling
with any $v=constant$ consist of a constant solution of (\ref{eq:1.2}).
This resembles the combustion wave equation considered in \cite{Ghazaryan},
however our method in proving the existence of the fronts of (\ref{eq:1.2})
differs from theirs. 

If the parameters satisfy 
\begin{equation}
0\leq\nu<K,\label{eq:1.3}
\end{equation}
then system (\ref{eq:1.2}) admits an additional equilibrium: $B:(1-\frac{\nu}{K},\frac{1}{K})$
representing the state that the spatial domain is successfully invaded.
We also separate the equilibrium $A:(0,0)$ from the rest of the line
of equilibria, $u=0$. The unstable equilibrium $(0,0)$ represents
the state before the invasion. 

We are interested in the existence of the wave solutions connecting
$A$ with $B$ as time and space evolve from $-\infty$ to $+\infty$.
Setting $\xi=x+ct$, $x\in\mathbb{R}$, $t\in\mathbb{R}^{+}$, a traveling
wave solution to (\ref{eq:1.2}) solves 

\begin{equation}
\left\{ \begin{array}{l}
u_{\xi\xi}-cu_{\xi}+u(1-u-\nu v)-\lambda u(1-Kv)=0,\\
\\
-cv_{\xi}+\lambda u(1-Kv)=0,
\end{array}\right.\label{eq:1.4}
\end{equation}
with boundary conditions:

\begin{equation}
(u,v)(-\infty)=(0,0),\quad(u,v)(+\infty)=(1-\frac{\nu}{K},\frac{1}{K}).\label{eq:1.5}
\end{equation}

For the notional convenience we further set

\[
\bar{u}=\frac{u}{1-\frac{\nu}{K}},
\]
and dropping the bar on $u$ to have

\begin{equation}
\left\{ \begin{array}{l}
u_{\xi\xi}-cu_{\xi}+u[1-\lambda-(1-\frac{\nu}{K})u+(\lambda K-\nu)v]=0,\\
\\
-cv_{\xi}+\lambda(1-\frac{\nu}{K})u(1-Kv)=0,
\end{array}\right.\label{eq:1.6}
\end{equation}
and 

\begin{equation}
(u,v)(-\infty)=(0,0),\quad(u,v)(+\infty)=(1,\frac{1}{K}).\label{eq:1.7}
\end{equation}

Numerical investigations \cite{Trewnack} strongly suggest that system
(\ref{eq:1.8}) and (\ref{eq:1.7}) admit traveling wave solutions
for $\nu=0$ and $\nu=1$. When the differentiated cell density does
not affect the proliferation of the the percursor cells, we have $\nu=0$;
and when the the total cell population contributes to the proliferation
carrying capacity, we have $\nu=1$. Numerically, however when $\nu=1$,
(\ref{eq:1.6}) may have non-monotone traveling wave solutions and
requires a different treatment. Hence in this paper we only study
the wave solutions for $\nu=0$. The system (\ref{eq:1.6}) in this
case can be further reduced to 

\begin{equation}
\left\{ \begin{array}{l}
u_{\xi\xi}-cu_{\xi}+u(1-\lambda-u+\lambda Kv)=0,\\
\\
-cv_{\xi}+\lambda u(1-Kv)=0.
\end{array}\right.\label{eq:1.8}
\end{equation}

The computations in \cite{Trewnack} shows that the wave may exist
for $c\geq2\sqrt{1-\lambda}$, but a rigorous existence proof is still
lacking. We will confirm this observation by a mathematical analysis
of the model. The system is of cooperative type, we can use the monotone
iteration scheme developed in \cite{FangZhao} for the existence proof.
Such method reduces the existence of the wave solutions to that of
the ordered upper and lower solution pairs for (\ref{eq:1.8}) and
(\ref{eq:1.7}). The upper and lower solutions in this paper come
straightly from two KPP type equations, which are so constructed that
they have the same decay rate at $-\infty$. Such information is also
relevent to the monotonicity and uniqueness of the wave solutions.
Indeed, since we have a good understanding of the decay properties
of the solutions at infinities, we then can study the properties of
the solutions on finite domain, in which the powerful sliding domain
method (see \cite{Berestycki}) can be used to have the desired results.
We remark that the methods we used in the proofs of the monotonicity
and the uniqueness have subtle difference from the ones used in \cite{LeungHou}.

\section{\label{sec:2}the main result.}

In this section we will use monotone iteration method to set up the
upper and lower solutions for system (\ref{eq:1.8}) and (\ref{eq:1.7}).
\begin{defn}
A $C^{2}(\mathbb{R})\times C^{1}(\mathbb{R})$ function $(\bar{u}(\xi),\bar{v}(\xi))^{T}$,
$\xi\in\mathbb{R}$ is an upper solution of (\ref{eq:1.8}) and (\ref{eq:1.7})
if it satisfies 
\begin{equation}
\left\{ \begin{array}{l}
u_{\xi\xi}-cu_{\xi}+u(1-\lambda-u+\lambda Kv)\leq0,\\
\\
-cv_{\xi}+\lambda u(1-Kv)\leq0
\end{array}\right.\label{eq:2.1}
\end{equation}
and the boundary conditions 
\begin{equation}
\left(\begin{array}{c}
u\\
\\
v
\end{array}\right)(-\infty)\geq\left(\begin{array}{c}
0\\
\\
0
\end{array}\right),\,\;\left(\begin{array}{c}
u\\
\\
v
\end{array}\right)(+\infty)\geq\left(\begin{array}{c}
1\\
\\
\frac{1}{K}
\end{array}\right).\label{eq:2.2}
\end{equation}

We can similarly define the lower solution $(\underline{u},\underline{v})(\xi)$,
$\xi\in\mathbb{R}$ by reversing the inequalities (\ref{eq:2.1})
and (\ref{eq:2.2}).
\end{defn}
The following known result (\cite{Sattinger}) is needed in the construction
of the upper and lower solutions:

Consider the following form of the KPP equation: 
\begin{equation}
\left\{ \begin{array}{l}
\omega''-c\omega'+f(\omega)=0,\\
\\
\omega(-\infty)=0,\quad\omega(+\infty)=b.
\end{array}\right.\label{eq:2.3}
\end{equation}
where $f\in C^{2}([0,\, b])$ and $f>0$ on the open interval $(0,b)$
with $f(0)=f(b)=0$, $f'(0)=\bar{a}>0$ and $f'(b)=-b_{1}<0$. 
\begin{lem}
\label{lem:KPP}Corresponding to every $c\geq2\sqrt{\bar{a}}$, system
(\ref{eq:2.3}) has a unique (up to a translation of the origin) monotonically
increasing traveling wave solution $\omega(\xi)$ for $\xi\in\mathbb{R}$.
The traveling wave solution $\omega$ has the following asymptotic
behaviors:

For the wave solution with non-critical speed $c>2\sqrt{\bar{a}}$,
we have 
\begin{equation}
\omega(\xi)=\bar{a}_{\omega}e^{\frac{c-\sqrt{c^{2}-4\bar{a}}}{2}\xi}+o(e^{\frac{c-\sqrt{c^{2}-4\bar{a}}}{2}\xi})\mbox{ as }\xi\rightarrow-\infty,\label{eq:2.4}
\end{equation}
 
\begin{equation}
\omega(\xi)=b-\bar{b}_{\omega}e^{\frac{c-\sqrt{c^{2}+4b_{1}}}{2}\xi}+o(e^{\frac{c-\sqrt{c^{2}+4b_{1}}}{2}\xi})\mbox{ as }\xi\rightarrow+\infty,\label{eq:2.5}
\end{equation}
 where $\bar{a}_{\omega}$ and $\bar{b}_{\omega}$ are positive constants.

For the wave with critical speed $c=2\sqrt{\bar{a}}$, we have

\begin{equation}
\omega(\xi)=\bar{d}_{c}\xi e^{\sqrt{\bar{a}}\xi}+o(\xi e^{\sqrt{\bar{a}}\xi})\mbox{ as }\xi\rightarrow-\infty,\label{eq:2.6}
\end{equation}

\begin{equation}
\omega(\xi)=b-\bar{b}_{c}e^{(\sqrt{\bar{a}}-\sqrt{\bar{a}+b_{1}})\xi}+o(e^{(\sqrt{\bar{a}}-\sqrt{\bar{a}+b_{1}})\xi})\mbox{ as }\xi\rightarrow+\infty,\label{eq:2.7}
\end{equation}
 where the constant $\bar{d}_{c}$ is negative, $\bar{b}_{c}$ is
positive. 
\end{lem}
We next consider the following version of the KPP system:

\begin{equation}
\left\{ \begin{array}{l}
\omega''-c\omega'+(1-\lambda)\omega(1-\omega)=0,\\
\\
\omega(-\infty)=0,\quad\omega(+\infty)=1.
\end{array}\right.\label{eq:2.8}
\end{equation}
According to Lemma \ref{lem:KPP}, for every $c\geq2\sqrt{1-\lambda}$
system (\ref{eq:2.8}) has a unique (up to a translation of the origin)
monotone solution $\tilde{u}(\xi)$, $\xi\in\mathbb{R}$. Now fix
this $\tilde{u}(\xi)$ and consider the equation

\begin{equation}
-cv_{\xi}+\lambda\tilde{u}(1-Kv)=0\label{eq:2.9}
\end{equation}

For each fixed $c\geq2\sqrt{1-\lambda}$ and the corresponding $\tilde{u}(\xi)$,
(\ref{eq:2.9}) has a solution

\begin{equation}
\bar{v}(\xi)=\frac{1}{K}(1-e^{-\frac{\lambda}{c}K\int_{-\infty}^{\xi}\tilde{u}(s)ds}).\label{eq:*}
\end{equation}

We next compare $\tilde{u}(\xi)$ with $\bar{v}(\xi)$ for $\xi\in\mathbb{R}$.
\begin{lem}
\label{lem:UpperComparison}There exists a $\zeta_{1}\geq0$ such
that if 

\begin{equation}
0<\lambda\leq\frac{2}{2+K(1+\sqrt{2})},\label{eq:&&}
\end{equation}
we have 

\[
\frac{1}{K}\tilde{u}(\xi+\zeta_{1})\geq\bar{v}(\xi)\qquad\xi\in\mathbb{R}.
\]
\end{lem}
\begin{proof}
According to Lemma \ref{lem:KPP}, the wave solution $\tilde{u}(\xi)$
to (\ref{eq:2.8}) has the following asymptotic behaviors: 

For $c>2\sqrt{1-\lambda}$, 
\begin{equation}
\tilde{u}(\xi)=a_{\omega}e^{\frac{c-\sqrt{c^{2}-4(1-\lambda)}}{2}\xi}+o(e^{\frac{c-\sqrt{c^{2}-4(1-\lambda)}}{2}\xi})\mbox{ as }\xi\rightarrow-\infty,\label{eq:2.11}
\end{equation}
 
\begin{equation}
\tilde{u}(\xi)=b-b_{\omega}e^{\frac{c-\sqrt{c^{2}+4(1-\lambda)}}{2}\xi}+o(e^{\frac{c-\sqrt{c^{2}+4(1-\lambda)}}{2}\xi})\mbox{ as }\xi\rightarrow+\infty,\label{eq:2.12}
\end{equation}
 and $a_{\omega}$, $b_{\omega}$ are positive constants.

For $c=2\sqrt{1-\lambda}$, we have

\begin{equation}
\tilde{u}(\xi)=d_{c}\xi e^{\sqrt{1-\lambda}\xi}+o(\xi e^{\sqrt{1-\lambda}\xi})\mbox{ as }\xi\rightarrow-\infty,\label{eq:2.13}
\end{equation}

\begin{equation}
\tilde{u}(\xi)=b-b_{c}e^{(1-\sqrt{2})\sqrt{1-\lambda}\xi}+o(e^{(1-\sqrt{2})\sqrt{1-\lambda}\xi})\mbox{ as }\xi\rightarrow+\infty.\label{eq:2.14}
\end{equation}
 where the constant $d_{c}$ is negative, $b_{c}$ is positive. 

We now study the asymptotics of the function $\bar{v}(\xi)$. Formulas
(\ref{eq:2.11}) and (\ref{eq:2.13}) imply that

\[
\int_{-\infty}^{\xi}\tilde{u}(s)ds\rightarrow0\quad\mbox{as}\:\xi\rightarrow-\infty.
\]
We can then expand 

\begin{equation}
e^{-\frac{\lambda}{c}K\int_{-\infty}^{\xi}\tilde{u}(s)ds}=1-\frac{\lambda}{c}K\int_{-\infty}^{\xi}\tilde{u}(s)ds+o((\int_{-\infty}^{\xi}\tilde{u}(s)ds)^{2}).\label{eq:2.15}
\end{equation}

A further expanding of (\ref{eq:2.15}) for $\xi\rightarrow-\infty$
and for $c>2\sqrt{1-\lambda}$,

\begin{equation}
1-e^{-\frac{\lambda}{c}K\int_{-\infty}^{\xi}\tilde{u}(s)ds}=\frac{2\lambda a_{\omega}}{(c-\sqrt{c^{2}-4(1-\lambda)})c}e^{\frac{c-\sqrt{c^{2}-4(1-\lambda)}}{2}\xi}+o(e^{\frac{c-\sqrt{c^{2}-4(1-\lambda)}}{2}\xi}),\label{eq:2.16}
\end{equation}
and for $c=2\sqrt{1-\lambda}$, 

\begin{equation}
1-e^{-\frac{\lambda}{c}K\int_{-\infty}^{\xi}\tilde{u}(s)ds}=\frac{\lambda}{2\sqrt{1-\lambda}}(d_{c}\xi-\frac{d_{c}}{1-\lambda})e^{\sqrt{1-\lambda}\xi}+o(\xi e^{\sqrt{1-\lambda}\xi}).\label{eq:2.17}
\end{equation}

As for $\xi>0$ sufficiently large we have

\begin{equation}
\lim_{\xi\rightarrow+\infty}\frac{\int_{-\infty}^{\xi}\tilde{u}(s)ds}{\xi}=\lim_{\xi\rightarrow+\infty}\tilde{u}(\xi)=1,\label{eq:2.18}
\end{equation}
therefore 

\begin{equation}
\bar{v}(\xi)=K(1-e^{-\frac{\lambda}{c}K\xi})+o(e^{-\frac{\lambda}{c}K\xi})\quad\mbox{as}\,\xi\rightarrow+\infty.\label{eq:2.19}
\end{equation}

We next show 
\begin{equation}
-\frac{\lambda}{c}K\geq\frac{c-\sqrt{c^{2}+4(1-\lambda)}}{2},\label{eq:2.20}
\end{equation}
or equivalently 

\[
\frac{2c(1-\lambda)}{c+\sqrt{c^{2}+4(1-\lambda)}}\geq\lambda K.
\]

Setting $g(c)=\frac{2(1-\lambda)}{1+\sqrt{1+4(1-\lambda)/c^{2}}}$,
then it is easy to see that $g(c)$ increases as $c$ does. Hence 

\[
g(c)\geq g(2\sqrt{1-\lambda})=\frac{2(1-\lambda)}{1+\sqrt{2}}.
\]

We therefore require $0<\lambda\leq\frac{2}{2+K(1+\sqrt{2})}$ to
have (\ref{eq:2.20}).

We now shift $\tilde{u}(\xi)$. Since (\ref{eq:2.8}) is shifting
invariant, $\tilde{u}(\xi+\zeta),$ $\xi\in\mathbb{R}$ is also a
solution for any $\zeta\in\mathbb{R}$. It then follows from (\ref{eq:2.11})
for $c>\sqrt{1-\lambda}$, 
\begin{equation}
\tilde{u}(\xi+\zeta)=a_{\omega}e^{\frac{c-\sqrt{c^{2}-4(1-\lambda)}}{2}\zeta}e^{\frac{c-\sqrt{c^{2}-4(1-\lambda)}}{2}\xi}+o(e^{\frac{c-\sqrt{c^{2}-4(1-\lambda)}}{2}\xi})\mbox{ as }\xi\rightarrow-\infty;\label{eq:2.21}
\end{equation}
and for $c=\sqrt{1-\lambda}$, 

\begin{equation}
\tilde{u}(\xi+\zeta)=d_{c}(\xi+\zeta)e^{\sqrt{1-\lambda}\zeta}e^{\sqrt{1-\lambda}\xi}+o(\xi e^{\sqrt{1-\lambda}\xi})\mbox{ as }\xi\rightarrow-\infty.\label{eq:2.22}
\end{equation}

If we choose $\zeta>0$ sufficiently large, the positiveness of $\frac{c-\sqrt{c^{2}-4(1-\lambda)}}{2}\zeta$
and $\sqrt{1-\lambda}\zeta$ implies that if $c=2\sqrt{1-\lambda}$, 

\[
d_{c}(\xi+\zeta)e^{\sqrt{1-\lambda}\zeta}>\frac{\lambda}{2\sqrt{1-\lambda}}(d_{c}\xi-\frac{d_{c}}{1-\lambda})
\]
and if $c>2\sqrt{1-\lambda}$, 

\[
a_{\omega}e^{\frac{c-\sqrt{c^{2}-4(1-\lambda)}}{2}\zeta}>\frac{2\lambda a_{\omega}}{(c-\sqrt{c^{2}-4(1-\lambda)})c}
\]

It then follows from (\ref{eq:2.20}), (\ref{eq:2.21}) and (\ref{eq:2.22})
that there exists a $\bar{N}>0$ sufficiently large, 

\[
\frac{1}{K}\tilde{u}(\xi+\zeta)\geq\bar{v}(\xi)\qquad\mbox{for }\xi\in(-\infty,-N]\cup[N,+\infty),
\]
and for $\xi\in[-N,N]$, since $\tilde{u}$ and $\bar{v}$ are both
monotonically increasing on $\mathbb{R}$ we can further shift $\tilde{u}(\xi+\zeta)$
to the left at most $2N$ units to have $\frac{1}{K}\tilde{u}(\xi+\zeta)\geq\bar{v}(\xi),\,\xi\in\mathbb{R}$.
Hence there exists a finite $\zeta_{1}\geq0$ such that the conclusion
of the Lemma holds.
\end{proof}
Now we write $\bar{u}(\xi)=\tilde{u}(\xi+\zeta_{0})$, $\xi\in\mathbb{R}$
and let $\bar{v}(\xi)$ be defined in (\ref{eq:*}). We remark here
that the computation of $\bar{v}(\xi)$ still uses $\tilde{u}(\xi)$.
\begin{lem}
\label{lem:Upper}Assume the conditions in Lemma \ref{lem:UpperComparison}
then $(\bar{u},\bar{v})(\xi)$, $\xi\in\mathbb{R}$ defines an upper
solution for (\ref{eq:1.8}) and (\ref{eq:1.7}).\end{lem}
\begin{proof}
We can easily verify that $(\bar{u},\bar{v})(\xi)$ satisfies the
boundary conditions (\ref{eq:2.2}).

For the $u$ component we have

\[
\begin{array}{cl}
 & \bar{u}''-c\bar{u}'+\bar{u}(1-\lambda-\bar{u}+\lambda K\bar{v})\\
\\
= & \bar{u}[1-\lambda-\bar{u}+\lambda K\bar{v}-(1-\lambda)(1-\bar{u})]\\
\\
= & -\lambda K\bar{u}(\frac{1}{K}\bar{u}-\bar{v})\leq0
\end{array}
\]
The last inequality follows from the previous Lemma.

As for the $v$ component, for each $\tilde{u}$, we have 

\[
-c\bar{v}_{\xi}+\lambda\tilde{u}(1-K\bar{v})=0.
\]

\end{proof}
We next set up the lower solution for (\ref{eq:1.8}) and (\ref{eq:1.7}).

For a fixed $l>0$ we consider another version of the KPP system:

\begin{equation}
\left\{ \begin{array}{l}
w''-cw'+(1-\lambda)w(1-\frac{1-\lambda+l}{1-\lambda}w)=0,\\
\\
w(-\infty)=0,\quad w(+\infty)=\frac{1-\lambda}{1-\lambda+l}<1.
\end{array}\right.\label{eq:2.23}
\end{equation}
Then for any $c\geq2\sqrt{1-\lambda}$, (\ref{eq:2.23}) has correspondingly
a unique wave solution $\breve{u}(\xi)$, $\xi\in\mathbb{R}$.

We define 

\[
\underline{v}(\xi)=\frac{1}{K}(1-e^{-\frac{\lambda}{c}K\int_{-\infty}^{\xi}\breve{u}(s)ds}).
\]

The next Lemma gives the relation between $\breve{u}(\xi)$ and $\underline{v}(\xi)$,
$\xi\in\mathbb{R}.$
\begin{lem}
\label{lem:lowercomparison}There exists a $\zeta_{1}\geq0$ such
that

\begin{equation}
\frac{1}{K}\breve{u}(\xi-\zeta_{1})\leq\underline{v}(\xi)\qquad\xi\in\mathbb{R}.\label{eq:2.24}
\end{equation}
\end{lem}
\begin{proof}
The proof is similar to that of Lemma \ref{lem:UpperComparison}.
Noting as $\xi\rightarrow+\infty$, $\breve{u}(\xi)\rightarrow\frac{1}{K}\frac{1-\lambda}{1-\lambda+l}<\frac{1}{K}$.
Hence we do not need condition (\ref{eq:&&}) here. 
\end{proof}
We denote $\underline{u}(\xi)=\breve{u}(\xi-\zeta_{1}),$ $\xi\in\mathbb{R}$.
Then 
\begin{lem}
\label{lem:lowersolution}Such defined $(\underline{u},\underline{v})(\xi)$,
$\xi\in\mathbb{R}$ consists of a lower solution for (\ref{eq:1.8})
and (\ref{eq:1.7}).\end{lem}
\begin{proof}
One the boundary we have 
\[
(\underline{u},\underline{v})(-\infty)=(0,0),\quad(\underline{u},\underline{v})(+\infty)=(\frac{1-\lambda}{1-\lambda+l},\frac{1}{K})\leq(1,\frac{1}{K}).
\]
and for the $u$ component,

\[
\begin{array}{cl}
 & \underline{u}''-c\underline{u}'+\underline{u}(1-\lambda-\underline{u}+\lambda K\underline{v})\\
\\
= & \underline{u}''-c\underline{u}'+(1-\lambda)\underline{u}(1-\frac{1-\lambda+l}{1-\lambda}\underline{u})-(1-\lambda)\underline{u}(1-\frac{1-\lambda+l}{1-\lambda}\underline{u})\\
\\
 & +\underline{u}(1-\lambda-\underline{u}+\lambda K\underline{v})\\
\\
= & \underline{u}[1-\lambda-\underline{u}+\lambda K\underline{v}-(1-\lambda)+(1-\lambda+l)\underline{u}]\geq0
\end{array}
\]
 due to the last Lemma.

Noting that $\underline{v}$ solves the equation 
\[
-c\underline{v}_{\xi}+\lambda\breve{u}(1-K\underline{v})=0,
\]
it satsfies the inequality trivially.\end{proof}
\begin{lem}
\label{lem:ordered }The upper and lower solutions are ordered
\begin{equation}
(\bar{u},\bar{v})(\xi)\geq(\underline{u},\underline{v})(\xi),\qquad\xi\in\mathbb{R}.\label{eq:2.25}
\end{equation}
\end{lem}
\begin{proof}
For each fixed $c\geq2\sqrt{1-\lambda}$ , if $\breve{u}(\xi)$ solves
the system (\ref{eq:2.23}) then the function $\tilde{u}(\xi)=\frac{1-\lambda+l}{1-\lambda}\breve{u}(\xi)$
solves (\ref{eq:2.8}). Hence it follows that $\tilde{u}(\xi)>\breve{u}(\xi)$
hence $\bar{u}(\xi)>\underline{u}(\xi)$ for all $\xi\in\mathbb{R}$. 

By the definition of $\bar{v}(\xi)$ and $\underline{v}(\xi)$, we
have 

\[
\begin{array}{lll}
\bar{v}(\xi)=\frac{1}{K}(1-e^{-\frac{\lambda}{c}K\int_{-\infty}^{\xi}\tilde{u}(s)ds}) & = & \frac{1}{K}(1-e^{-\frac{\lambda}{c}K\frac{1-\lambda+l}{1-\lambda}\int_{-\infty}^{\xi}\breve{u}(s)ds})\\
\\
 & > & \frac{1}{K}(1-e^{-\frac{\lambda}{c}K\int_{-\infty}^{\xi}\breve{u}(s)ds})=\underline{v}(\xi).
\end{array}
\]

Hence the conclusion of the Lemma holds.\end{proof}
\begin{thm}
Let the parameters satisfy (\ref{eq:&&}), then for each $c\geq2\sqrt{1-\lambda}$,
system (\ref{eq:1.8}) and (\ref{eq:1.7}) has a unique (up to a translation)
strictly monotonically increasing traveling wave solution, while for
$0<c<2\sqrt{1-\lambda}$, there is no monotonic traveling wave. The
Traveling wave solution has the following asymptotic behaviors:

For $c=2\sqrt{1-\lambda}$

\begin{equation}
\left(\begin{array}{c}
u\\
\\
v
\end{array}\right)(\xi)=\left(\begin{array}{c}
c_{11}\xi\\
\\
c_{12}\xi
\end{array}\right)e^{\sqrt{1-\lambda}\xi}+o(e^{\sqrt{1-\lambda}\xi}),\qquad\mbox{as }\xi\rightarrow-\infty,\label{eq:1}
\end{equation}
and 
\begin{equation}
\left(\begin{array}{c}
u\\
\\
v
\end{array}\right)(\xi)=\left(\begin{array}{c}
1\\
\\
\frac{1}{K}
\end{array}\right)-\left(\begin{array}{c}
c_{21}e^{-\frac{\lambda K}{2\sqrt{1-\lambda}}\xi}\\
\\
c_{22}e^{-\frac{\lambda K}{2\sqrt{1-\lambda}}\xi}
\end{array}\right)+\left(\begin{array}{c}
o(e^{-\frac{\lambda K}{2\sqrt{1-\lambda}}\xi})\\
\\
o(e^{-\frac{\lambda K}{2\sqrt{1-\lambda}}\xi})
\end{array}\right),\qquad\mbox{as }\xi\rightarrow+\infty;\label{eq:2}
\end{equation}
and for $c>2\sqrt{1-\lambda}$

\begin{equation}
\left(\begin{array}{c}
u\\
\\
v
\end{array}\right)(\xi)=\left(\begin{array}{c}
d_{11}\xi\\
\\
d_{12}\xi
\end{array}\right)e^{\frac{c-\sqrt{c^{2}+4(1-\lambda)}}{2}\xi}+o(e^{\frac{c-\sqrt{c^{2}+4(1-\lambda)}}{2}\xi}),\qquad\mbox{as }\xi\rightarrow-\infty,\label{eq:3}
\end{equation}
and

\begin{equation}
\left(\begin{array}{c}
u\\
\\
v
\end{array}\right)(\xi)=\left(\begin{array}{c}
1\\
\\
\frac{1}{K}
\end{array}\right)-\left(\begin{array}{c}
d_{21}e^{-\frac{\lambda K}{c}\xi}\\
\\
d_{22}e^{-\frac{\lambda K}{c}\xi}
\end{array}\right)+\left(\begin{array}{c}
o(e^{-\frac{\lambda K}{c}\xi})\\
\\
o(e^{-\frac{\lambda K}{c}\xi})
\end{array}\right)\qquad\mbox{as }\xi\rightarrow+\infty,\label{eq:4}
\end{equation}
where $c_{11},c_{12},c_{21},c_{22},d_{21},d_{22}>0$, $d_{11},d_{12}<0$. \end{thm}
\begin{proof}
Noting that between the upper and lower solutions, there is no equilibrium
other than $(0,0)$ and $(1,\frac{1}{K})$ of system (\ref{eq:1.8})
and (\ref{eq:1.7}). Hence the monotone iteration scheme developed
in \cite{FangZhao} is still applicable. Such monotone iteration scheme
reduces the existence of the traveling wave solutions to that of the
ordered upper and lower solution pairs, the existence of the traveling
waves then follows by Lemma \ref{lem:lowersolution}, Lemma \ref{lem:Upper}
and Lemma \ref{lem:ordered }, and by \cite{FangZhao}, such obtained
traveling wave solutions are nondecreasing. While for $c<2\sqrt{1-\lambda}$
it is easy to verify, by analyzing the equilibrium $(0,0)$ that the
nontrivial bounded solutions of (\ref{eq:1.8}) are oscillatory.

We next show that the wave solutions are strictly monotonically increasing
on $\mathbb{R}$. 

For any fixed $c\geq\sqrt{1-\lambda}$, let $(u_{c},v_{c})(\xi)$
be the corresponding traveling wave solution, and $(w_{1}(\xi),w_{2}(\xi))$
be its derivative. Then $(w_{1}(\xi),w_{2}(\xi))\geq0$ for $\xi\in\mathbb{R}$,
and $(w_{1}(\xi),w_{2}(\xi))$ satisfies the following systems

\begin{equation}
\left\{ \begin{array}{l}
w_{1,\xi\xi}-cw_{1,\xi}+(1-\lambda-2u_{c}+\lambda Kv_{c})w_{1}+\lambda Ku_{c}w_{2}=0,\\
\\
-cw_{2,\xi}+\lambda(1-Kv_{c})w_{1}-\lambda Ku_{c}w_{2}=0,\\
\\
(w_{1},w_{2})(\pm\infty)=0.
\end{array}\right.\label{eq:2.27}
\end{equation}
It then follows that 

\begin{equation}
\left\{ \begin{array}{l}
w_{1,\xi\xi}-cw_{1,\xi}+(1-\lambda-2u_{c}+\lambda Kv_{c})w_{1}\leq0,\\
\\
-cw_{2,\xi}+\lambda(1-Kv_{c})w_{1}-\lambda Ku_{c}w_{2}=0,\\
\\
(w_{1},w_{2})(\pm\infty)=0.
\end{array}\right.\label{eq:2.28}
\end{equation}

Applying the Maximum Principle to the first inequality of (\ref{eq:2.28}),
we immediately conclude that $w_{1}(\xi)>0$ for $\xi\in\mathbb{R}$.
Thus $u_{c}(\xi)$ is strictly monotonically increasing.

The strict monotonicity of $v_{c}(\xi)$ comes from (\ref{eq:1.8}).
Since $u_{c}(\xi)>0$ for all $\xi\in\mathbb{R}$, and for such $u_{c}(\xi)$
we have

\[
v_{c}(\xi)=\frac{1}{K}(1-e^{-\frac{\lambda}{c}K\int_{-\infty}^{\xi}u_{c}(s)ds})
\]
then it follows that $w_{2}(\xi)=v_{c}'(\xi)>0$, $\xi\in\mathbb{R}$.
This shows that the wave solution $(u_{c},v_{c})$ is strictly monotonically
increasing.

We then derive the asymptotics of the wave solutions at $\pm\infty$.
Noting that the upper and lower solutions have the same exponential
decay rate at $-\infty$, (\ref{eq:1}) and (\ref{eq:3}) come directly
from comparison.

We next study the asymptotics of the function $(w_{1},w_{2})(\xi)$
at $+\infty$, recalling that $(w_{1},w_{2})(\xi)=(u_{c},v_{c})'(\xi)$
and satisfies the system (\ref{eq:2.27}). Since this system is hyperbolic
at $+\infty$, $(w_{1},w_{2})$ approaches $(0,0)$ exponentially.
We will derive the exact exponential rate. 

The limit equation at $+\infty$ of system (\ref{eq:2.27}) is 

\[
\left\{ \begin{array}{l}
w_{1,\xi\xi}^{+}-cw_{1,\xi}^{+}-w_{1}^{+}+\lambda Kw_{2}^{+}=0,\\
\\
-cw_{2,\xi}^{+}-\lambda Ku_{c}w_{2}^{+}=0
\end{array}\right.
\]

Since the second equation is decoupled from the system, we immediately
have 

\[
w_{2}^{+}(\xi)=\underline{A}e^{-\frac{\lambda K}{c}\xi}
\]
Plugging the above into the first equation yields a bounded solution
(up to the first order approximation) of the form 

\[
w_{1}^{+}(\xi)=\bar{A}_{1}e^{-\frac{\lambda K}{c}\xi}+\bar{A}_{2}e^{\frac{c-\sqrt{c+4}}{2}\xi}.
\]

By roughness of exponential dichotomy \cite{1-coppel}, we have 

\[
\left(\begin{array}{c}
w_{1}(\xi)\\
\\
w_{2}(\xi)
\end{array}\right)=\left(\begin{array}{c}
\underline{A_{1}}e^{-\frac{\lambda K}{c}\xi}\\
\\
\bar{A}_{2}e^{\mu\xi}
\end{array}\right)+\left(\begin{array}{c}
o(e^{-\frac{\lambda K}{c}\xi})\\
\\
o(e^{\mu\xi})
\end{array}\right)
\]
where $\mu$ is either $-\frac{\lambda K}{c}$ or $\frac{c-\sqrt{c+4}}{2}$. 

Integrating the above from $\xi_{0}$ to $+\infty$, and comparing
the decay rates of $(u_{c},v_{c})(\xi)$ with that of the upper solution
$(\bar{u},\bar{v})(\xi)$, we have (\ref{eq:2}) and (\ref{eq:4}). 

On the uniqueness of the traveling wave solution for every $c\geq2\sqrt{1-\lambda}$,
we only prove the conclusion for traveling wave solutions with asymptotic
rates given in (\ref{eq:3}) and (\ref{eq:4}) since the other case
can be proved similarly. Let $U_{1}(\xi)=(u_{1},v_{1})(\xi)$ and
$U_{2}(\xi)=(u_{2},v_{2})(\xi)$ be two traveling wave solutions of
system (\ref{eq:1.8}) and (\ref{eq:1.7}) with the same speed $c>2\sqrt{1-\lambda}$.
There exist positive constants $A_{ij}$, $B_{ij}$, $i,j=1,2$ and
a large number $N>0$ such that for $\xi<-N$,
\begin{equation}
U_{1}(\xi)=\left(\begin{array}{c}
A_{11}\\
\\
A_{12}
\end{array}\right)e^{\frac{c+\sqrt{c^{2}-4(1-\lambda)}}{2}\xi}+o(e^{\frac{c+\sqrt{c^{2}-4(1-\lambda)}}{2}\xi})\label{eq:2.32}
\end{equation}
 
\begin{equation}
U_{2}(\xi)=\left(\begin{array}{c}
A_{21}\\
\\
A_{22}
\end{array}\right)e^{\frac{c+\sqrt{c^{2}-4(1-\lambda)}}{2}\xi}+o(e^{\frac{c+\sqrt{c^{2}-4(1-\lambda)}}{2}\xi});\label{eq:2.33}
\end{equation}
 and for $\xi>N$,
\begin{equation}
U_{1}(\xi)=\left(\begin{array}{c}
{\displaystyle 1-B_{11}e^{-\frac{\lambda}{c}K\xi}}\\
\\
{\displaystyle \frac{1}{K}-B_{12}e^{-\frac{\lambda}{c}K\xi}}
\end{array}\right)+\left(\begin{array}{c}
{\displaystyle o(e^{-\frac{\lambda}{c}K\xi}})\\
\\
{\displaystyle (e^{-\frac{\lambda}{c}K\xi})}
\end{array}\right),\label{eq:2.34}
\end{equation}
 
\begin{equation}
U_{2}(\xi)=\left(\begin{array}{c}
{\displaystyle 1-B_{21}e^{-\frac{\lambda}{c}K\xi}}\\
\\
{\displaystyle \frac{1}{K}-B_{22}e^{-\frac{\lambda}{c}K\xi}}
\end{array}\right)+\left(\begin{array}{c}
{\displaystyle o(e^{-\frac{\lambda}{c}K\xi}})\\
\\
{\displaystyle (e^{-\frac{\lambda}{c}K\xi})}
\end{array}\right).\label{eq:2.35}
\end{equation}
 The traveling wave solutions of system (\ref{eq:1.8})-(\ref{eq:1.7})
are translation invariant, thus for any $\theta>0$, $U_{1}^{\theta}(\xi):=U_{1}(\xi+\theta)$
is also a traveling wave solution of (\ref{eq:1.8})-(\ref{eq:1.7}).
By (\ref{eq:2.32}) and (\ref{eq:2.34}), the solution $U_{1}(\xi+\theta)$
has the asymptotics 
\begin{equation}
U_{1}^{\theta}(\xi)=\left(\begin{array}{c}
A_{11}e^{\frac{c+\sqrt{c^{2}-4(1-\lambda)}}{2}\theta}\\
\\
A_{12}e^{\frac{c+\sqrt{c^{2}-4(1-\lambda)}}{2}\theta}
\end{array}\right)e^{\frac{c+\sqrt{c^{2}-4(1-\lambda)}}{2}\xi}+o(e^{\frac{c+\sqrt{c^{2}-4(1-\lambda)}}{2}\xi})\label{eq:2.36}
\end{equation}
 for $\xi\leq-N$;
\begin{equation}
U_{1}^{\theta}(\xi)=\left(\begin{array}{c}
{\displaystyle 1-B_{11}e^{-\frac{\lambda}{c}K\theta}e^{-\frac{\lambda}{c}K\xi}}\\
\\
{\displaystyle \frac{1}{K}-B_{12}e^{-\frac{\lambda}{c}K\theta}e^{-\frac{\lambda}{c}K\xi}}
\end{array}\right)+\left(\begin{array}{c}
{\displaystyle o(e^{-\frac{\lambda}{c}K\xi}})\\
\\
{\displaystyle (e^{-\frac{\lambda}{c}K\xi})}
\end{array}\right)\label{eq:2.37}
\end{equation}
 for $\xi\geq N$.

Choosing $\theta>0$ large enough such that 
\begin{equation}
A_{11}e^{\frac{c+\sqrt{c^{2}-4(1-\lambda)}}{2}\theta}>A_{21},\label{eq:2.38}
\end{equation}
 
\begin{equation}
A_{12}e^{\frac{c+\sqrt{c^{2}-4(1-\lambda)}}{2}\theta}>A_{22},\label{eq:2.39}
\end{equation}
 
\begin{equation}
B_{11}e^{-\frac{\lambda}{c}K\theta}<B_{21},\label{eq:2.40}
\end{equation}
 
\begin{equation}
B_{12}e^{-\frac{\lambda}{c}K\theta}<B_{22}.\label{eq:2.41}
\end{equation}
 then one has for $\xi\in(-\infty,-N]$$\cup$$[N,+\infty),$ 
\begin{equation}
U_{1}^{\theta}(\xi)>U_{2}(\xi).\label{eq:2.42}
\end{equation}

We now consider system (\ref{eq:1.8}) on $[-N,+N]$. There are two
possibilities:

Case 1. Suppose we already ahve $U_{1}^{\theta}(\xi)\geq U_{2}(\xi)$
on $[-N,+N]$, then the function $W(\xi)=(w_{1}(\xi),w_{2}(\xi))^{T}:=U_{1}^{\theta}(\xi)-U_{2}(\xi)\geq0$
and satisfies for some $\zeta_{i}\in(0,1)$, $i=1,2$,

\begin{equation}
\left\{ \begin{array}{l}
\left(\begin{array}{c}
w_{1}''\\
\\
0
\end{array}\right)-c\left(\begin{array}{c}
w_{1}'\\
\\
w_{2}'
\end{array}\right)+M\left(\begin{array}{c}
w_{1}\\
\\
w_{2}
\end{array}\right)=0\\
\\
W(-N)>0,\,\,\, W(+N)>0.\qquad\xi\in(-N,N),
\end{array}\right.\label{eq:2.43}
\end{equation}
where the matrix $M$ is given by 

\begin{equation}
M(w_{1},w_{2})=\left(\begin{array}{cc}
1-\lambda-2(u_{2}+\zeta_{1}w_{1})+\lambda K(v_{2}+\zeta_{2}w_{2}), & \lambda K(u_{2}+\zeta_{1}w_{1})\\
\\
\lambda(1-K(v_{2}+\zeta_{2}w_{2})), & -\lambda K(u_{2}+\zeta_{2}w_{1})
\end{array}\right).\label{eq:2.44}
\end{equation}

Since $w_{1}(\xi)\geq0$, $\xi\in[-N,N]$ and $\lambda K(u_{2}+\zeta_{1}w_{1})\geq0$,
then we have on $\xi\in[-N,N],$

\begin{equation}
\left\{ \begin{array}{l}
w_{1}''-cw_{1}'+[1-\lambda-2(u_{2}+\zeta_{1}w_{1})+\lambda K(v_{2}+\zeta_{2}w_{2})]w_{1}+\lambda K(u_{2}+\zeta_{2}w_{1})w_{2}=0,\\
\\
w_{1}(-N)>0,\qquad w_{1}(N)>0.
\end{array}\right.\label{eq:2.45}
\end{equation}

The Maximum Principle then implies that $w_{1}(\xi)>0$ on $[-N,N]$.
We then move to the second equation of (\ref{eq:2.43}). We have 

\begin{equation}
\left\{ \begin{array}{l}
-cw_{2}'-\lambda K(u_{2}+\zeta_{2}w_{1})w_{2}=-\lambda(1-K(v_{2}+\zeta_{2}w_{2}))w_{1}<0,\qquad\xi\in[-N,N],\\
\\
w_{2}(-N)>0,\qquad w_{2}(N)>0.
\end{array}\right.\label{eq:2.46}
\end{equation}

The strict inequality comes from the fact that $v_{2}(\xi)\leq v_{2}(\xi)+\zeta_{2}w_{2}(\xi)\leq v_{1}^{\theta}<\frac{1}{K}$
for $\xi\in[-N,N]$. It then follows that $w_{2}(\xi)>0$ for $\xi\in[-N,N]$.
For if there is a $\bar{\xi}\in(-N,N)$ such that $w_{2}(\bar{\xi})=0.$Then
$w_{2}$ takes local minimum at $\bar{\xi}$, then the left hand side
of the first inequality of (\ref{eq:2.46}) is zero at $\bar{\xi}$.
We then have a contradiction. 

Case two. We may suppose that there is some point in $(-N,N)$ such
that one of the components, say the $j$-th component, satisfies $(U_{1}^{\theta}(\xi))_{j}<(U_{2}(\xi))_{j}$
at that point, $j=1$ or $2$. We then increase $\theta$, that shifts
$U_{1}^{\theta}(\xi)$ further left, so that $U_{1}^{\theta}(-N)>U_{2}(-N)$,
$U_{1}^{\theta}(N)>U_{2}(N)$. By the monotonicity of $U_{1}^{\theta}$
and $U_{2}$, we can find a $\bar{\theta}\in(0,2N)$ such that in
the interval $(-N,N)$, we have $U_{1}^{\theta}(\xi+\bar{\theta})>U_{2}(\xi)$.
Shifting $U_{1}^{\theta}(\xi+\bar{\theta})$ back until one component
of $U_{1}^{\theta}(\xi+\bar{\theta})$ first touches its counterpart
of $U_{2}(\xi)$ at some point $\bar{\bar{\xi}}\in[-N,N]$. We then
return back to case 1 again, where it has been shown that this is
impossible. Therefore, we must have
\[
U_{1}^{\theta}(\xi)>U_{2}(\xi)
\]
 for all $\xi\in\mathbb{R}$, where $\theta$ is the one chosen by
means of (\ref{eq:2.38})-(\ref{eq:2.41}) as described above.

Now, decrease $\theta$ until one of the following situations happens.

1. There exists a $\bar{\theta}\geq0$, such that $U_{1}^{\bar{\theta}}(\xi)\equiv U_{2}(\xi)$.
In this case we have finished the proof.

2. There exists a $\bar{\theta}\geq0$ and $\xi_{1}\in\mathbb{R}$,
such that one of the components of $U^{\bar{\theta}}$ and $U_{2}$
are equal there; and for all $\xi\in\mathbb{R}$, we have $U_{1}^{\bar{\theta}}(\xi)\geq U_{2}(\xi)$.
On applying the Maximum Principle on $\mathbb{R}$ and use the same
argument as we did for case 1. We see this is impossible. 

Consequently, in either situation, there exists a $\bar{\theta}\geq0$,
such that 
\[
U_{1}^{\bar{\theta}}(\xi)\equiv U_{2}(\xi).
\]
 for all $\xi\in\mathbb{R}$. \end{proof}

\end{document}